\documentclass[11pt,leqno,twoside]{article}
\usepackage{multirow} 
\usepackage{mathrsfs,amssymb,amsfonts,amsthm,amsmath,natbib,amsbsy}
\usepackage{dsfont,verbatim,color,fancyhdr}
\usepackage{graphicx,float} 
\usepackage[hang]{caption} 
\usepackage[normalem]{ulem}
\usepackage[bookmarks,bookmarksnumbered,colorlinks=true,pdfstartview=FitV, linkcolor=red,citecolor=blue,urlcolor=blue]{hyperref}

\usepackage{pdfsync}

\allowdisplaybreaks

\bibpunct{\textcolor{blue}{(}}{\textcolor{blue}{)}}{;}{a}{\textcolor{blue}{,}}{\textcolor{blue}{;}}

\long\def\symbolfootnote[#1]#2{\begingroup\def\thefootnote{\hspace*{-1mm}\fnsymbol{footnote}}\footnote[#1]{#2}\endgroup}

\def \d {{\rm d}}

\def \A {\mathcal{A}}
\def \B {\mathcal{B}}

\def \D {\mathscr{D}}

\def \P {\mathscr{P}}
\def \PD {\mathrm{PD}}

\def \T {\mathcal{T}}
\def \Dn {\Delta_{n}}

\def \TDn {\Delta_{n,\varepsilon_{n}}}

\def \TNn {\nabla_{n,\varepsilon_{n}}}
\def \Ni {\overline\nabla_{\infty}}

\newtheorem{theorem}{Theorem}[section]
\newtheorem{proposition}[theorem]{Proposition} 

\newtheorem{lemma}[theorem]{Lemma} 
 
\newtheorem{definition1}[theorem]{Definition} 
 
\newtheorem{remark1}[theorem]{Remark} 

\title{\bf \vspace{-2.5cm} 
Species dynamics in the two-parameter Poisson--Dirichlet diffusion model
}
\author{\sc Matteo Ruggiero\\[1mm]
\em University of Torino and Collegio Carlo Alberto}

\fancypagestyle{plain}{
\fancyhead{}
}

\begin{document}

\maketitle
\thispagestyle{empty}

\symbolfootnote[0]{Webpage: \href{http://web.econ.unito.it/ruggiero}{http://web.econ.unito.it/ruggiero}}

\vspace{-9mm}
\begin{center}
\begin{minipage}{.75\textwidth}
\footnotesize\quad\ 
The recently introduced two-parameter infinitely-many neutral alleles model extends the celebrated one-parameter version, related to Kingman's distribution, to diffusive two-parameter Poisson--Dirichlet frequencies. Here we investigate the dynamics driving the species heterogeneity underlying the two-parameter model. First we show that a suitable normalization of the number of species is driven by a critical continuous-state branching process with immigration. Secondly, we provide a finite-dimensional construction of the two-parameter model, obtained by means of a sequence of Feller diffusions of Wright--Fisher flavor which feature finitely-many types and inhomogeneous mutation rates. Both results provide insight into the mathematical properties and biological interpretation of the two-parameter model, showing that it is structurally different from the one-parameter case in that the frequencies dynamics are driven by state-dependent rather than constant quantities.
\\[-2mm]


\end{minipage}
\end{center}


\section{Introduction}
The two-parameter infinitely-many neutral alleles model is a family of infinite dimensional diffusion processes, 
introduced by \cite{P09} and further investigated by \cite{RW09} and \cite{FS10}, that extends the celebrated one-parameter version, formulated by \cite{W76} and characterized by \cite{EK81}. Throughout the paper we will refer to the one- and two-parameter infinitely-many neutral alleles model simply as the one- and two-parameter model. 
More specifically, let 
\begin{equation}\label{nabla-infty}
\overline\nabla_{\infty}=\bigg\{z\in[0,1]^{\infty}:\,z_{1}\ge z_{2}\ge\ldots\ge0,\;\sum_{i=1}^{\infty}z_{i}\le1\bigg\}
\end{equation}
be the closure of the infinite-dimensional ordered simplex, and define, for constants $0\le\alpha<1$ and $\theta>-\alpha$,
 the second-order differential operator 
\begin{align}\label{operator: theta-sigma}
\mathcal{B}
=\frac{1}{2}\sum_{i,j=1}^{\infty}z_{i}(\delta_{ij}-z_{j})\frac{\partial^{2}}{\partial z_{i}\partial z_{j}}-
\frac{1}{2}\sum_{i=1}^{\infty}(\theta z_{i}+\alpha)\frac{\partial}{\partial z_{i}},
\end{align}
where $\delta_{ij}$ denotes Kronecker delta,
acting on a certain dense subalgebra of the space $C(\overline\nabla_{\infty})$ of continuous functions on $\overline\nabla_{\infty}$. Then  the closure of $\mathcal{B}$ generates a strongly continuous semigroup of contractions on $C(\overline\nabla_{\infty})$, and the sample paths of the associated process are almost surely continuous functions from $[0,\infty)$ to $\overline\nabla_{\infty}$. Such process can be thought of as describing the temporal evolution of the decreasingly ordered allelic frequencies $(z_{1},z_{2},\ldots)$ at a particular locus in an ideally infinite population with infinitely-many possible types or species. 
See \cite{F10} for a review of infinite alleles models. Recently, 
\cite{Fetal11} determined the transition density for the two-parameter case. See also \cite{BO09} for a general construction related to \cite{P09}, and \cite{RWF13} for a partially related model with diffusive parameter $\theta$.

As shown by \cite{P09} and \cite{FS10}, the two-parameter model is reversible and ergodic with respect to the Poisson--Dirichlet distribution with parameters $(\theta,\alpha)$, henceforth denoted $\PD(\theta,\alpha)$. Introduced by \cite{PPY92} (see also \citealp{P95} and \citealp{PY97}), this extends the one-parameter version $\PD(\theta):=\PD(\theta,0)$ due to \cite{K75}, and 
has found numerous applications in several fields. See for example \cite{B06} for fragmentation and coalescent theory, \cite{P06} for excursion theory and combinatorics, \cite{A08} for economics, \cite{LP09} for Bayesian inference, \cite{TJ09} for machine learning and \cite{F10} for population genetics.

Both these random discrete distributions arise as the decreasingly ordered weights of a Dirichlet process \citep{F73}, when $\alpha=0$, and of a two-parameter Poisson--Dirichlet (or Pitman-Yor) process \citep{P95} respectively. Alternatively, they can be constructed by means of the following so-called stick breaking procedure, also known as residual allocation model. Consider a sequence of random variables $(V_{1},V_{2},\ldots)$ obtained by setting
\begin{equation*}
V_{1}=W_{1},
\qquad
V_{n}=W_{n}\prod_{i=1}^{n-1}(1-W_{i}),
\qquad
W_{i}\overset{\text{ind}}{\sim}\text{Beta}(1-\alpha,\theta+i\alpha),
\end{equation*} 
where `$\overset{\text{ind}}{\sim}$' denotes independence, $0\le\alpha<1$ and $\theta>-\alpha$.
The vector $(V_{1},V_{2},\ldots)$ is said to have the GEM$(\theta,\alpha)$ distribution, named after Griffiths, Engen and McCloskey, while the vector of descending order statistics $(V_{(1)},V_{(2)},\ldots)$ is said to have the $\PD(\theta,\alpha)$ distribution. See \cite{FW07} for an infinite-dimensional diffusion process related to GEM distributions.

Besides sharing the above stick-breaking construction strategy, it is well known that the difference between these two random discrete distributions is structural and does not simply rely on a different parametrization. 
For example, the distribution $\PD(\theta)$ can be obtained by ranking and normalizing the jumps of a Gamma subordinator, whereas the $\PD(\theta,\alpha)$ is obtained by performing the same operation on the jumps of a stable subordinator and appropriately mixing over the law of the normalizing factor \citep{P03}. See Section \ref{subsection: 2par}.  Furthermore, the $\PD(\theta)$ distribution is obtained as weak limit of a Dirichlet distributed vector of frequencies \citep{K75}, while a similar construction for the two-parameter case is not available. For what concerns their diffusive counterparts, the properties of the one-parameter model, related to the $\PD(\theta)$ distribution, are well understood, whereas numerous are still the questions regarding the two-parameter model. In particular, given the above considerations, it is  not surprising that a finite-dimensional construction of the process with operator (\ref{operator: theta-sigma}), in terms of a sequence of finite dimensional diffusion processes, is currently available only when $\alpha=0$. 
To be more precise, consider the usual approximating diffusion for the Wright--Fisher discrete genetic model with $n$ selectively neutral alleles and symmetric mutation. This corresponds to the operator 
\begin{align}\label{B-n}
\mathcal{B}_{n}
=&\,\frac{1}{2}\sum_{i,j=1}^{n}z_{i}(\delta_{ij}-z_{j})\frac{\partial^{2}}{\partial z_{i}\partial z_{j}}
+\frac{1}{2}\sum_{i=1}^{n}b_{i}^{(n)}(z)\frac{\partial}{\partial z_{i}},
\end{align}
acting on a suitable subspace of $C^{2}(\nabla_{n})$, with
\begin{equation}
\nabla_{n}=\bigg\{z\in\Ni:\  z_{n+1}=0,\ \sum_{i=1}^{n}z_{i}=1\bigg\},\notag
\end{equation} 
with drift components
\begin{equation}\label{drift-one-par}
b_{i}^{(n)}(z)=\frac{\theta}{n-1}(1-z_{i})-\theta z_{i},
\quad \quad \theta>0.
\end{equation} 
\cite{EK81} formalized the conditions under which the sequence of processes with operators defined by (\ref{B-n})-(\ref{drift-one-par}) converges in distribution to the one-parameter model, with operator obtained by setting $\alpha=0$ in (\ref{operator: theta-sigma}). 
As anticipated, a similar construction for the case $0<\alpha<1,\ \theta>-\alpha$ is currently unavailable. It is to be said that two different sequential constructions of the two-parameter model are given in \cite{P09} and \cite{RW09}. In Section \ref{sec: technical complement}  we will argue that despite offering interesting reads of the two parameter model, neither of these provides particular insight for the interpretation of the species dynamics underlying the infinite-dimensionality structure. In particular this is due to the fact that both are based on finitely-many items. The problem at hand could then be rephrased as that of understanding from which Wright--Fisher-type mechanism, if any, the two-parameter model comes from. 
While the importance of providing a particle construction lies in the fact that the individual dynamics are dealt with explicitly, the contribution of a sequential construction by means of Wright--Fisher-type diffusions lies in the genetic interpretation one yields from the specification of the mutation rates at the $n$th step of the sequence. 
Such interpretation is clear in the case of (\ref{drift-one-par}), whereby each type has the same chance of mutating (cf.~also (\ref{symm-mut}) below), but is somewhat obscure for what regards the role of $\alpha$ in (\ref{operator: theta-sigma}), especially in terms of its effect on finitely-many types. 
This role would be, at least partially, clarified by identifying suitable mutation rates, which are basic building blocks of the model and give important information on the reproductive mechanism of the underlying population. 
Historically, the (chrono)logical process has been the opposite, namely diffusion approximations were introduced for dealing mathematically in a simpler way with multi-type discrete models such as  Wright--Fisher processes. But recent advances, stimulated by neighboring research fields, have provided the infinite-dimensional diffusion without identifying its finite-dimensional source, thus leaving an interpretational gap.

Motivated by these considerations, the purpose of this paper is to investigate what lies underneath the infinite-dimensionality of the two-parameter model in terms of the forces driving the species dynamics. We pursue this task in two different ways. First we derive an $\alpha$-diversity diffusion for the two-parameter model. This is a continuous-time continuous-state extension of the corresponding notion for Poisson-Kingman models \citep{P03}, and describes the dynamics of the suitably normalized number of species in the underlying population. In Section 2 we show that such diffusion for the two-parameter model is a critical continuous-state branching process with immigration, and we discuss a corresponding quantity for the one-parameter case. 
Second, we find explicit transition rates for the mutation process which gives rise to the two-parameter model, and provide a sequential construction for the limiting process in terms of finite-dimensional diffusions similar to those given by (\ref{B-n})-(\ref{drift-one-par}) for the one-parameter case. 
In Section \ref{sec: technical complement} we collect some brief considerations on the problem and the fact that the existing constructions do not provide enough insight from a biological point of view. 
In Section \ref{sec:sequential-construction} we identify 
mutation rates that yield the convergence result; these turn out to
depend on the current species abundances. By means of some additional restrictions, we formalize a sequential construction where each term of the sequence is a Feller diffusion on a finite-dimensional subspace of $\Ni$. 

In achieving the two aforementioned goals, we are able to highlight a key difference between the one- and two-parameter model, conveniently summarized by saying that the species dynamics of the former are driven by constant terms, whereas those of the latter are driven by density-dependent quantities.


\section{Heterogeneity in the two-parameter model}\label{subsection: 2par}

The notion of $\alpha$-diversity was introduced by \cite{P03} for exchangeable partitions induced by random discrete distributions of Poisson-Kingman type. Let $\textsc{pk}(\varrho|t)$ denote the distribution of the weights $(P_{i})_{i\ge1}=(J_{i}/T)_{i\ge1}$ determined by the ranked points $J_{i}$ of a Poisson process with L\'evy density $\varrho$, given $T=t$.
A Poisson-Kingman distribution with L\'evy density $\varrho$ and mixing distribution $\gamma$ on $(0,\infty)$, denoted $\textsc{pk}(\varrho,\gamma)$, is defined as the mixture
\begin{equation*}
\textsc{pk}(\varrho,\gamma)=\int_{0}^{\infty}\textsc{pk}(\varrho|t)\gamma(\d t).
\end{equation*} 
For instance, the $\PD(\theta,\alpha)$ distribution is obtained as a $\textsc{pk}(\varrho_{\alpha},\gamma_{\theta,\alpha})$ model, for $0<\alpha<1$ and $ \theta>-\alpha$, where $\varrho_{\alpha}$ is the L\'evy density of a stable subordinator of index $\alpha$, $\gamma_{\theta,\alpha}$ is
\begin{equation*}
\gamma_{\theta,\alpha}(\d t)=\frac{\Gamma(\theta+1)}{\Gamma(\theta/\alpha+1)}
t^{-\theta}f_{\alpha}(t)\d t,
\end{equation*} 
and $f_{\alpha}(t)$ is the density of a positive stable random variable of index $\alpha$.
Given an exchangeable random partition of $\mathbb{N}$ induced by a  Poisson-Kingman distribution, i.e.~such that its ranked class frequencies have distribution $\textsc{pk}(\varrho,\gamma)$, this is said to have $\alpha$-diversity $S$ if and only if there exists a random variable $S$, with $0<S<\infty$ almost surely, such that
\begin{equation}\label{alpha-diversity}
\lim_{n\rightarrow\infty}\frac{K_{n}}{n^{\alpha}}= S\qquad \text{a.s.},
\end{equation} 
where $K_{n}$ is the number of classes of the partition restricted to $\{1,\ldots,n\}$. For instance, in the case of a $\PD(\theta,\alpha)$ partition, we have $S=T^{-\alpha}$ where $T$ has distribution $\gamma_{\theta,\alpha}$. See \cite{P03}, Proposition 13.

The idea of extending the concept of $\alpha$-diversity from random distributions on simplices to a continuous-time continuous-state framework has been formulated in \cite{RWF13}, where a certain rescaled, inhomogeneous random walk on the integers, which tracks the dynamics of the number of species in a normalized inverse-Gaussian population, is shown to converge to a certain one dimensional diffusion process on $(0,\infty)$.
Here we derive an $\alpha$-diversity diffusion for the two-parameter model, with the aim of providing insight into the species dynamics underlying the infinite-dimensional process.
The derivation is based on the particle construction given in \cite{RW09}, here briefly recalled for ease of the reader. Let $X^{(n)}=(X_{1},\ldots,X_{n})$ be a sample from a two-parameter Poisson--Dirichlet process, or equivalently (cf.~\citealp{P95}) from the generalized P\'olya urn scheme given by $X_{1}\sim P_{0}$ and
\begin{equation}\label{urn}
\mathbb{P}\{X_{i+1}\in\cdot|X_{1},\dots,X_{i}\}=
\frac{\theta+\alpha K_{i}}{\theta+i}P_0(\cdot)+\frac{1}{\theta+i}\sum_{j=1}^{K_{i}}(n_j-\alpha)\,\delta_{X_j^*}(\cdot)
\end{equation} 
for $i=2,\ldots,n-1$. Here $P_{0}$ is a non atomic probability measure on the space of the observables (e.g.~a Polish space), $K_{i}\le i$ denotes the number of distinct elements $(X_{1}^{*},\dots,X_{K_{i}}^{*})$ observed in $(X_{1},\dots,X_{i})$ and $\delta_{X_j^*}$ is a point mass at $X_j^*$. A simple way to make the sample into a Markov chain with fixed marginals is the following. 
Let $X^{(n)}$ be updated at discrete times by replacing a uniformly chosen 
coordinate. Conditionally on $K_{n}=k$ at the current state, and exploiting the exchangeability of the sample, the incoming particle will be of a new type with probability $(\theta+\alpha k_{r})/(\theta+n-1)$, and will be a copy of one still in 
the vector after the removal with probability $(n-1-\alpha k_{r})/(\theta+n-1)$, where $k_{r}$ is the value of $k$ after the removal. 

The following proposition recalls, in a discrete parameter version, the relation between the above described particle chain $\{X^{(n)}(m),m\in\mathbb{N}_{0}\}$ and the two-parameter model.
For notational simplicity we omit here the details about the domain of the limiting operator (cf.~\eqref{symm polyn}-\eqref{domain B con phi-2} below). Here and throughout `$\Rightarrow$' denotes convergence in distribution and $C_{A}([0,\infty))$ denotes the space of continuous functions from $[0,\infty)$ to the space $A$.

\begin{proposition}\emph{[\citealp*{RW09}]}\label{prop: particle conv}
Let $Z(\cdot)$ be the two-parameter model, corresponding to the operator $\B$ as in (\ref{operator: theta-sigma}). Let also $\{X^{(n)}(m),m\in\mathbb{N}_{0}\}$ be the particle chain described above, and define $Y^{(n)}(\cdot)=\{Y^{(n)}(t),t\ge0\}$ by $Y^{(n)}(t)=\eta(X^{(n)}(\lfloor n^{2}t\rfloor))$, where $\eta(x^{(n)})=(z_{1},\ldots,z_{n},0,0,\ldots)$ if $z_{i}$ is the relative size of the $i$-th largest cluster in $x^{(n)}$. Then 
\begin{equation*}
 Y^{(n)}(\cdot) \Rightarrow Z(\cdot),
\quad 
\text{in }C_{\overline\nabla_{\infty}}([0,\infty)),\notag
\end{equation*} 
as $n\rightarrow\infty$.
\end{proposition}

Hence the Markov chain $\{X^{(n)}(m),m\in\mathbb{N}_{0}\}$, once appropriately transformed and rescaled, provides a Moran-type particle construction of the two-parameter model.
Denote now by $\{K_{n}(m),m\in\mathbb{N}_{0}\}$ the chain which keeps track of the number of distinct types in $X^{(n)}(m)$, and let $M_{1,n}$ denote the number of types in $X^{(n)}(m)$ with only one representative. The transition 
probabilities for $K_{n}(m)$, denoted for short
\begin{equation*}
p(k,k')=\mathbb{P}\{K_{n}(m+1)=k'|K_{n}(m)=k\}
\end{equation*}
are given by 
\begin{align}\label{K-transition-0}
p(k,k')=&\,\left\{
\begin{array}{ll}
\displaystyle \left(1-\frac{M_{1,n}}{n}\right)\frac{\theta+\alpha k}{\theta+n-1},  &   \text{if }k'=k+1,\\[3mm] 
\displaystyle \frac{M_{1,n}}{n(\theta+n-1)}(n-1-\alpha(k-1)),  \qquad&   \text{if }k'=k-1,\\[3mm]
1-p(k,k+1)-p(k,k-1), & \text{if }k'=k,\\[2mm]
0,&\text{else},
\end{array}
\right.
\end{align}
for $1\le k\le n$. Here $M_{1,n}/n$ is the probability of removing a cluster of size one, and $k=1$ and $k=n$ imply $p(1,0)=0$ and $p(n,n+1)=0$ respectively. Since (\ref{K-transition-0}) need not be Markovian, we use an approximation of $M_{1,n}$ based on the following asymptotic result.
From (\ref{alpha-diversity}) and Lemma 3.11 in \cite{P06}, we have that the number $M_{1,n}$ of clusters of size one observed in the sample is such that
\begin{equation}\label{M-convergence}
\frac{M_{1,n}}{n^{\alpha}}\rightarrow \alpha S\qquad \text{a.s.},\notag
\end{equation} 
so that $M_{1,n}\approx \alpha k$. This yields
\begin{align}\label{K-transition}
p(k,k')=&\,\left\{
\begin{array}{ll}
\displaystyle \left(1-\frac{\alpha k}{n}\right)\frac{\theta+\alpha k}{\theta+n-1}+o(n^{-1+\alpha}),  \quad \quad &   \text{if }k'=k+1,\\[3mm]
\displaystyle \frac{\alpha k}{n(\theta+n-1)}(n-1-\alpha(k-1))+o(n^{-1+\alpha}),  \qquad&   \text{if }k'=k-1,\\[3mm]
1-p(k,k+1)-p(k,k-1), & \text{if }k'=k,\\[2mm]
0,&\text{else}.
\end{array}
\right.
\end{align}
The following theorem identifies the $\alpha$-diversity diffusion for the two-parameter model.
Denote by $C_{0}([0,\infty))$ the space of continuous functions on $[0,\infty)$ vanishing at infinity.

\begin{theorem}\label{prop: alpha-div-conv-2par}
Let $\{K_{n}(m),m\in\mathbb{N}_{0}\}$ be a Markov chain on $\mathbb{N}$ with transition probabilities as in  (\ref{K-transition}), for $0<\alpha<1$ and $\theta>-\alpha$, and define $\{\tilde K_{n}(t),t\ge0\}$ by letting $\tilde K_{n}(t)=K_{n}(\lfloor n^{1+\alpha}t\rfloor )/n^{\alpha}$. Let also $\{S_{\theta,\alpha}(t),t\ge0\}$ be a diffusion process on $[0,\infty)$ driven by the stochastic differential equation
\begin{equation}\label{alpha-div-diffusion-2par}
\d S_{\theta,\alpha}(t)=\theta\d t+\sqrt{2\alpha S_{\theta,\alpha}(t)}\,\d B(t),
\end{equation}
where $B(t)$ is a standard Brownian motion.
If $\tilde K_{n}(0)\Rightarrow S_{\theta,\alpha}(0)$, then
\begin{equation}\label{weak-conv-S-2par}\notag
\{\tilde K_{n}(t),t\ge0\}\Rightarrow\{S_{\theta,\alpha}(t),t\ge0\},
\qquad 
\text{ in } C_{[0,\infty)}([0,\infty)),
\end{equation} 
as $n\rightarrow\infty$.
\end{theorem}
\begin{proof}
Denote by $U_{n}$ the semigroup induced by (\ref{K-transition}). For notational brevity, here we do not distinguish between $(n,k)$ and $(n-1,k-1)$, since they are asymptotically equivalent. Then, for $f\in C_{0}([0,\infty))$, we can write
\begin{align*}
(U_{n}-I)f\bigg(\frac{k}{n^{\alpha}}\bigg)
=&\,\left[f\bigg(\frac{k+1}{n^{\alpha}}\bigg)-f\bigg(\frac{k}{n^{\alpha}}\bigg)\right]\bigg(1-\frac{\alpha k}{n}\bigg)\frac{\theta+\alpha k}{\theta+n}\\
&\,+\left[f\bigg(\frac{k-1}{n^{\alpha}}\bigg)-f\bigg(\frac{k}{n^{\alpha}}\bigg)\right]
\frac{\alpha k(n-\alpha k)}{n(\theta+n)}+o(n^{-1+\alpha}).
\end{align*}
By means of a Taylor expansion we get
\begin{align}
\notag
(U_{n}-I)&\,f\bigg(\frac{k}{n^{\alpha}}\bigg)
=\frac{1}{n^{\alpha}}C^{(1)}_{\theta,\alpha,k,n}f'\bigg(\frac{k}{n^{\alpha}}\bigg)
+\frac{1}{2n^{2\alpha}}C^{(2)}_{\theta,\alpha,k,n}f''\bigg(\frac{k}{n^{\alpha}}\bigg)
+o(n^{-1+\alpha}),
\end{align}
where
\begin{equation*}
C^{(1)}_{\theta,\alpha,k,n}=\bigg(1-\frac{\alpha k}{n}\bigg)\frac{\theta+\alpha k}{\theta+n}
-\frac{\alpha k}{n}\bigg(\frac{n-\alpha k}{\theta+n}\bigg)
=\frac{\theta}{\theta+n}+o(n^{-1})
\end{equation*} 
and
\begin{equation*}
C^{(2)}_{\theta,\alpha,k,n}=\bigg(1-\frac{\alpha k}{n}\bigg)\frac{\theta+\alpha k}{\theta+n}
+\frac{\alpha k}{n}\bigg(\frac{n-\alpha k}{\theta+n}\bigg)
=\frac{2\alpha k}{\theta+n}+o(n^{-1+\alpha}).
\end{equation*} 
Using (\ref{alpha-diversity}), it follows that
\begin{equation}\label{operator convergence alpha}
\sup_{s\in[0,\infty)}|\mathcal{L}f(s)-n^{1+\alpha}(U_{n}-I)f(s)|\rightarrow0,
 \qquad \text{as }n\rightarrow\infty,
\end{equation} 
where $\mathcal{L}f(s)=\theta f'(s)+\alpha sf''(s)$ is the infinitesimal operator corresponding to (\ref{alpha-div-diffusion-2par}). Here (\ref{operator convergence alpha}) holds for every $f$ belonging to an appropriate restriction $\mathscr{D}(\mathcal{L})$ of $C_{0}([0,\infty))$ (to be formalized in Proposition \ref{prop:boundaries} below). Under these conditions,
Theorem 1.6.5 in \cite{EK86} implies that
\begin{equation}\label{semigroup convergence}
\sup_{s\in[0,\infty)}|U(t)f(s)-U_{n}(\lfloor n^{1+\alpha}t\rfloor)f(s)|\rightarrow0, 
\quad 
f\in C_{0}([0,\infty)),
\end{equation} 
as $n\rightarrow\infty$ and for all $t\ge0$, 
where $U$ is the semigroup operator corresponding to $\mathcal{L}$. 
The assertion of the theorem with $C_{[0,\infty)}([0,\infty))$ replaced by $D_{[0,\infty)}([0,\infty))$ now follows from \eqref{semigroup convergence} and 
Theorem 4.2.6 in \cite{EK86}. 
Finally, the convergence holds in $C_{[0,\infty)}([0,\infty))$ since the limit probability measure puts mass one on $C_{[0,\infty)}([0,\infty))$, and the Skorohod topology relativized to $C_{[0,\infty)}([0,\infty))$ coincides with the uniform topology of $C_{[0,\infty)}([0,\infty))$ (cf.~\citealp{B68}, Section 18).
\end{proof}

Hence the dynamic heterogeneity of the two-parameter model is described by a non negative diffusion obtained with a space-time rescaling which depends on the parameter $\alpha$.
Note that $S_{\theta,\alpha}(\cdot)$ in (\ref{alpha-div-diffusion-2par}) can be seen as
a critical continuous-state branching process with immigration \citep*{KW71,L06}, obtained for example as diffusion approximation of a Galton-Watson branching process with immigration, with unitary mean number of offspring per individual. Here $\theta>0$ is interpreted as the immigration rate; the case $\theta<0$ has been treated in \cite{GY03}.	

The next proposition, which provides the complete boundary behaviour of the process driven by (\ref{alpha-div-diffusion-2par}) and formalizes its well-definedness, is not new and included for formal completeness. 
Let $\mathcal{L}$ be the second order differential operator
\begin{equation}\label{S-generator}
\mathcal{L}=\theta\frac{\d}{\d s}+\alpha s\frac{\d^{2}}{\d s^{2}},
\qquad 0<\alpha<1,\ \theta>-\alpha.
\end{equation}

\begin{proposition}\label{prop:boundaries}
The process $\{S_{\theta,\alpha}(t),t\ge0\}$ driven by (\ref{alpha-div-diffusion-2par}) has the following boundary behavior: the boundary $s=0$ is absorbing for $\theta\le 0$, instantaneously reflecting for $0<\theta< \alpha$, and entrance for $\theta\ge\alpha$;
 the boundary $s=\infty$ is natural and non attracting for $\theta\le \alpha$, and natural and attracting for $\theta>\alpha$.
Moreover, $S_{\theta,\alpha}(t)$ is null recurrent for $\theta=\alpha$ and transient for $\theta\ne \alpha$.
For $\mathcal{L}$ as in (\ref{S-generator}), define
\begin{equation*}
\mathscr{D}(\mathcal{L})=
\Big\{f\in C_{0}([0,\infty))\cap C^{2}((0,\infty)):\ \mathcal{L}f\in C_{0}([0,\infty))\Big\},
\end{equation*} 
and
\begin{equation*}
\mathscr{D}_{\theta,\alpha}(\mathcal{L})=
\left\{
\begin{array}{ll}
f\in \mathscr{D}(\mathcal{L}),  & \text{if } \theta\ge\alpha, \\[3mm]
f\in \mathscr{D}(\mathcal{L}): \lim_{x\rightarrow0}x^{\theta/\alpha}f'(x)=0,
\quad  & \text{if } 0<\theta<\alpha,\\[3mm]
f\in \mathscr{D}(\mathcal{L}): \lim_{x\rightarrow0}\mathcal{L}f(x)=0,
\quad  & \text{if } -\alpha<\theta\le 0.
\end{array}
\right.
\end{equation*}
Then $\{(f,\mathcal{L}f):f\in\mathscr{D}_{\theta,\alpha}(\mathcal{L})\}$ generates a Feller semigroup on $C_{0}([0,\infty))$.
\end{proposition}
\begin{proof}
The first assertion follows from \cite{IW89}, Example IV.8.2 and \cite{KT81}, Table 15.6.2.
The second assertion follows from the first assertion and \cite{IW89}, Theorem VI.3.1.
The third assertion follows from the first assertion, together with the fact that $\exp\{\int_{1}^{x}(\theta/\alpha y)\d y\}=x^{\theta/\alpha}$, and with Theorem 8.1.1 and Corollary 8.1.2 in \cite{EK86}. 
\end{proof}

The lack of positive recurrence immediately determines the non stationarity of the process $\{S_{\theta,\alpha}(t),t\ge0\}$.

We conclude the section with a brief discussion of the corresponding process for the one-parameter model. 
Although the notion of $\alpha$-diversity is given for Poisson-Kingman models with $0<\alpha<1$ (cf.~\citealp{P03}), a result analogous to Theorem \ref{prop: alpha-div-conv-2par} can be nonetheless derived for the one-parameter case, for which $\alpha=0$. 
The limit corresponding to (\ref{alpha-diversity}) when $\alpha=0$ is provided by \cite{KH73} and is
\begin{equation*}
\lim_{n\rightarrow\infty}\frac{K_{n}}{\log n}= \theta,\qquad \text{a.s.}
\end{equation*} 
Hence we expect the process for the normalized number of species to converge to a constant process, i.e.
\begin{equation}\label{norm-Kn-1par}\notag
\{K_{n}(\lfloor c_{n}t\rfloor)/\log n,t\ge0\}\overset{p}{\longrightarrow}S_{\theta,0}(t)\equiv\theta,
\end{equation} 
for some $c_{n}\rightarrow\infty$. Setting $\alpha=0$ in \eqref{K-transition-0} and proceeding similarly to the proof of Theorem \ref{prop: alpha-div-conv-2par}, we get
\begin{align*}
(U_{n}&-I)f\bigg(\frac{k}{\log n}\bigg)
=\frac{1}{\log n}f'\bigg(\frac{k}{\log n}\bigg)
\bigg[\bigg(1-\frac{w}{n}\bigg)\frac{\theta}{\theta+n}
-\frac{w}{n}\bigg(\frac{n}{\theta+n}\bigg)\bigg]\notag\\
&\,+\frac{1}{2\log^{2} n}f''\bigg(\frac{k}{\log n}\bigg)
\bigg[\bigg(1-\frac{w}{n}\bigg)\frac{\theta}{\theta+n}
+\frac{w}{n}\bigg(\frac{n}{\theta+n}\bigg)\bigg]+o((n\log n)^{-1})
\end{align*}
where $w$ stands for the fact that
\begin{equation*}
M_{1,n}\overset{d}{\longrightarrow} W\sim\text{Poisson}(\theta);
\end{equation*} 
 see \cite{ABT92}. Hence in the limit the argument of the derivatives is constant, and $c_n(U_{n}-I)f(k/\log n)$, with $c_{n}=n\log n$, converges to 0.
It follows that the dynamics of the number of species underlying the infinite-alleles models are driven by the constant process $S_{\theta,0}(t)\equiv\theta$ in the one-parameter case, and by the diffusion process $S_{\theta,\alpha}(t)$ on $[0,\infty)$ with state-dependent volatility in the two-parameter case. This confirms the structural difference between one- and two-parameter models also from this dynamic viewpoint. A similar difference between the two cases will be found again in Section \ref{sec:sequential-construction} with a different approach.


\section{Finite-dimensional construction of the two-parameter model}

\subsection{Preliminary remarks}\label{sec: technical complement}

In the Introduction it was mentioned that two different sequential constructions of the two-parameter model have been provided in \cite{P09} and \cite{RW09}. In this section we briefly outline why these offer only partial insight into the dynamics underlying the two-parameter model from a biological perspective, motivating the need for further investigation. 

The above-mentioned constructions are given respectively by a sequence taking values in the space of partitions of $\mathbb{N}$, and by the Moran-type particle representation outlined in Section \ref{subsection: 2par} above. Both cases are based on a dynamic system of finitely-many exchangeable particles and exhibit right-continuous sample paths,
whereas (\ref{B-n}) (with an appropriate domain) characterizes an $n$-dimensional diffusion process. 
Another notable feature of these constructions is the assumption that the distribution that generates the mutant types is nonatomic and thus selects types which appear for the first time with probability one (in the framework of \cite{P09} this amounts to say that a new box is occupied with probability one). In particular, such feature turns out to be the key for proving the weak convergence of the sequences to the two-parameter model
 (see e.g.~\cite{RW09} after Remark 3.1). 
Such assumption of non atomicity cannot be applied in a construction similar to (\ref{B-n})-(\ref{drift-one-par}), because the mass of the distribution must concentrate on the enumerated types, in order to keep the maximum amount of species constant in time. 
To be more precise about this aspect, note first that the drift coefficients in the Wright--Fisher operator (\ref{B-n}) are determined as 
\begin{equation}\label{drift decomposition}
b^{(n)}_{i}(z)=\sum_{j\ne i}q^{(n)}_{ji}(z)z_{j}-\sum_{j\ne i}q^{(n)}_{ij}(z)z_{i}.
\end{equation} 
Here $q_{ij}$ is the intensity of a mutation from type $i$ to type $j$, and diagonal elements are $q_{ii}=-\sum_{j\ne i}^{n}q_{ij}$, so that $(q_{ij})_{i,j=1,\ldots,n}$ is a square matrix with nonnegative off-diagonal elements and row sums equal to zero. In general the mutation rate $q^{(n)}_{ij}(z)$ can be thought of as state-dependent, but in many interesting cases only the dependence on $n$ is needed.
The drift (\ref{drift-one-par}), for example, is obtained by taking parent-independent symmetric mutations with rates
\begin{equation}\label{symm-mut}
q_{ij}^{(n)}=\frac{\theta}{n-1},\qquad \qquad  i\ne j,
\end{equation} 
whereby when a type $i$ mutates, the new type will be any of the other $n-1$ types with equal chances, and $\theta$ controls how often on average mutations occur. In this case then the mutant type is chosen with uniform probability, and the mutant type distribution is discretely supported. Such derivation of the one-parameter model can be extended to have non-symmetric mutation (see for example \cite{EK81}, Theorem 3.4), but the difference is not relevant for our purposes.

Hence the two existing constructions for the two-parameter model feature finitely-many objects, potentially of infinitely-many types, and a diffuse mutant type distribution, while the desired construction should feature infinitely-many objects of finitely-many types and a discretely supported mutant type distribution.

From a mathematical point of view, ideally we seek mutation rates $q^{(n)}_{ij}(z)$ yielding, through \eqref{drift decomposition}, the $i$th component limit drift term
\begin{equation}\label{limit drift}
b^{(n)}_{i}(z)\overset{\text{unif}}{\longrightarrow}-\theta z_{i}-\alpha,
\end{equation} 
and the $b^{(n)}_{i}(z)$'s satisfy the boundary conditions 
 \begin{equation}\label{boundary conditions}
b_{i}^{(n)}(z)\ge 0,\quad  \text{if }z_{i}=0,\qquad\quad
b_{i}^{(n)}(z)\le 0,\quad  \text{if }z_{i}=1,
\end{equation} 
for $z\in \Delta_{n}$ and
\begin{equation}\label{Delta-n}
\Delta_{n}=\left\{z\in[0,1]^{n}:\ z_{i}\ge0,\ \sum_{i=1}^{n}z_{i}=1\right\}.
\end{equation} 
However, obtaining \eqref{limit drift} and \eqref{boundary conditions} jointly is clearly not possible, since in $z_{i}=0$ the drift should be non negative for all $n$ but strictly negative in the limit. 
Since condition \eqref{boundary conditions} is crucial for the well-definedness of the $n$th term of the sequence, the alternative strategy will then be to relax \eqref{limit drift} to the weaker condition
\begin{equation}\label{drifts sum convergence}
\sum_{i=1}^{n}b^{(n)}_{i}(z)\frac{\partial f(z)}{\partial z_{i}}\overset{\text{unif}}{\longrightarrow}-\sum_{i=1}^{\infty}(\theta z_{i}+\alpha)\frac{\partial f(z)}{\partial z_{i}},
\end{equation} 
as $n\rightarrow\infty$, for a sufficiently large set of functions $f(z)$. Obtaining rates $q_{ij}^{(n)}(z)$ which yields drift terms satisfying \eqref{drifts sum convergence}, together with some additional restrictions concerning the volatility and the state space of the process, will then suffice to provide the desired convergence.


\subsection{Sequential construction}\label{sec:sequential-construction}

Let $n\ge2$ throughout the section, and let $\Dn$ be as in (\ref{Delta-n}).
Consider a sequence of real numbers $\{\varepsilon_{n}\}_{n\in\mathbb{N}}$ satisfying
\begin{equation}\label{epsilon n}
0<\varepsilon_{n}<\frac{1}{n}\ \ \forall n,
\qquad 
\varepsilon_{n}=o(n^{-1}),
\end{equation} 
and define the compact subspace of $\Dn$ given by
\begin{equation*}
\TDn=\left\{z\in[0,1]^{n}:\ z_{i}\ge\varepsilon_{n}, \sum_{i=1}^{n}z_{i}=1\right\},
\end{equation*} 
where $z\in\TDn$ implies $z_{i}\in[\varepsilon_{n},1-(n-1)\varepsilon_{n}]\ne \emptyset$ for all $i$. See Figure \ref{fig: simplex} below.
Consider the second order differential operator
\begin{equation}\label{operator An}
\A_{n}
=\frac{1}{2}\sum_{i,j=1}^{n}a^{(n)}_{ij}(z)\frac{\partial^{2}}{\partial z_{i}\partial z_{j}}
+\frac{1}{2}\sum_{i=1}^{n}b_{i}^{(n)}(z)\frac{\partial}{\partial z_{i}},
\end{equation} 
with domain 
\begin{equation}\label{domain-An}
\D(\A_{n})=\left\{f:\ f\in C^{2}(\TDn)\right\},
\end{equation} 
where 
\begin{equation*}
C^{2}(\TDn)=\left\{f\in C(\TDn):\ \exists \tilde f\in C^{2}(\mathbb{R}^{n}),\ \tilde f|_{\TDn}=f\right\}.
\end{equation*} 
The covariance components in (\ref{operator An}) are specified to be
\begin{align}\label{cov-n}
a^{(n)}_{ij}(z)
=&\,(z_{i}-\varepsilon_{n})(\delta_{ij}(1-n\varepsilon_{n})-(z_{j}-\varepsilon_{n})) \\
=&\,
\left\{
\begin{array}{ll}
(z_{i}-\varepsilon_{n})(1-(n-1)\varepsilon_{n}-z_{i}),  \qquad   & \text{if }i=j,  \\[1mm]
-(z_{i}-\varepsilon_{n})(z_{j}-\varepsilon_{n}),   & \text{if }i\ne j.
\end{array}
\right.\notag
\end{align}
These can be seen as Wright--Fisher-type covariance terms restricted to $\TDn$, since they vanish at $z_{i}=\varepsilon_{n}$ and $z_{i}=1-(n-1)\varepsilon_{n}$. 
Additionally, consider the state-dependent mutation rates 
\begin{equation}\label{mut rates}
q_{ij}^{(n)}(z)=\frac{\theta}{n-1}+\frac{2\alpha j}{z_{i}n(n+1)}
\left[1-\exp\left\{-2(z_{i}-\varepsilon_{n})/\varepsilon_{n}\right\}\right], \qquad i\ne j.
\end{equation}  

Before providing some considerations on the form of $q_{ij}^{(n)}(z)$, note that (\ref{drift decomposition}) yields the drift components
\begin{align}\label{drift-n}
\begin{array}{rl}
b_{i}^{(n)}(z)
=&\!\displaystyle \frac{\theta}{n-1}(1-z_{i})-\theta z_{i}
+\frac{2\alpha i}{n(n+1)}
\sum_{j=1}^{n}
\left[1-\exp\left\{-2(z_{j}-\varepsilon_{n})/\varepsilon_{n}\right\}\right]\\
&-\alpha \left[1-\exp\left\{-2(z_{i}-\varepsilon_{n})/\varepsilon_{n}\right\}\right].
\end{array}
\end{align}
Here the first two terms of $b_{i}^{(n)}(z)$ equal
\begin{equation*}
\frac{\theta}{n-1}(1-(n-1)\varepsilon_{n}-z_{i})-\theta (z_{i}-\varepsilon_{n}),
\end{equation*} 
and $z_{i}=1-(n-1)\varepsilon_{n}$ implies $z_{j}=\varepsilon_{n}$ for all $j\ne i$. Using the last two observations in \eqref{drift-n} shows that
\begin{equation}\label{drift at boundary}
b_{i}^{(n)}(z)>0\quad \text{if }z_{i}=\varepsilon_{n},
\qquad \qquad 
b_{i}^{(n)}(z)<0\quad \text{if }z_{i}=1-(n-1)\varepsilon_{n},
\end{equation} 
so that $b_{i}^{(n)}(z)$ satisfies (\ref{boundary conditions}) restricted to $\TDn$.

In order to provide some intuition on \eqref{mut rates}, we have to consider separately the constant and the frequency-dependent term. The former attributes equal chances of mutation to all species regardless of their abundance, as in the one-parameter model. To evaluate the effect of $q_{ij}^{(n)}(z)$ as a deviation from (\ref{symm-mut}), recall now that the limit operator (\ref{operator: theta-sigma}) acts on functions defined on (\ref{nabla-infty}), where the frequencies have been ordered. The same ordering operation will be done before taking the limit of $\A_{n}$, with the formal appearance of the operator unchanged, so that it is correct to think in terms of ranked frequencies. In light of this, the term $j$ at the numerator of the second term in $q_{ij}^{(n)}(z)$  can be interpreted as an approximate indication of the size of the frequency $z_{j}$, a greater $j$ implying a lower $z_{j}$. 
Hence mutations from $i$ to $j$ occur more frequently if $z_{j}$ is relatively low, implying a redistributive effect. This has to be interpreted as a conditional mechanism, related to the probability of directing the mass of the $i$th-type individual to some species $j\ne i$, conditional on the fact that such individual mutates.
In order to evaluate the unconditional chances of mutation of $i$th-type individuals, consider now the rescaled state dependent term in $q_{ij}^{(n)}(z)$, namely
\begin{equation}\label{z part}
z_{i}^{-1}[1-\exp\{-2(z_{i}-\varepsilon_{n})/\varepsilon_{n}\}].
\end{equation} 
Recall that the range of values of $z_{i}$ is determined by $n$ through $\varepsilon_{n}$ and  grows to $[0,1]$ as $n$ diverges, and note that the clear non monotonicity of the quantity in \eqref{z part} is displayed on such range only for $n$ large enough.
\begin{figure}[t!]
\begin{center}
\includegraphics[width=.7\textwidth]{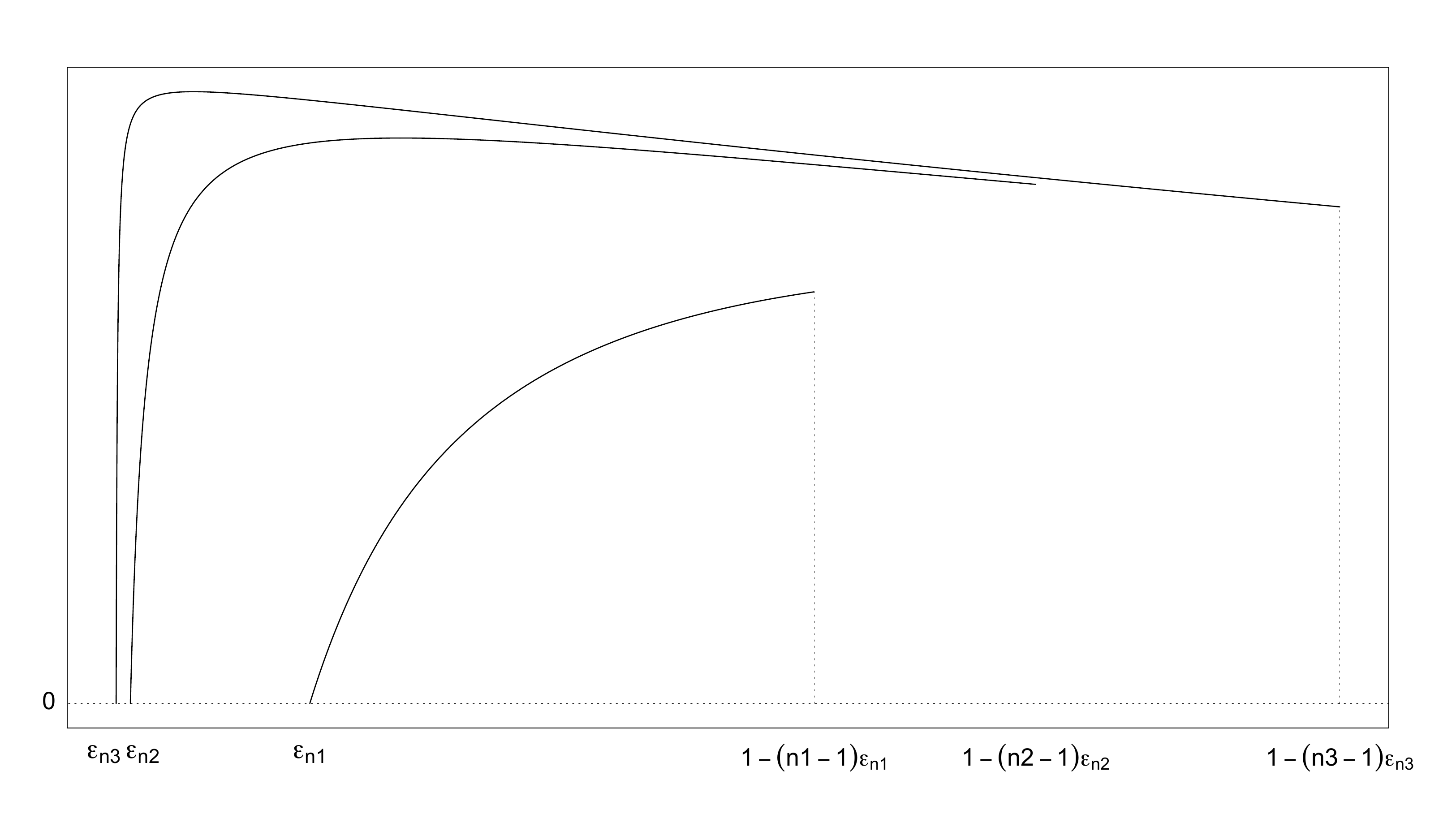}\\
\begin{quote}
\caption{\footnotesize 
Contribution of the (rescaled) state-dependent part of the mutation rate $q_{ij}^{(n)}(z)$ in (\ref{mut rates}) with respect to the rate of the one-parameter model. The plot shows 
a qualitative comparison of (\ref{z part}) as a function of $z_{i}$ with  $\varepsilon_{n}=n^{-1.1}$ and $n_{1}<n_{2}<n_{3}$.}\label{fig: rate}
\end{quote}\end{center}
\end{figure} 
Figure \ref{fig: rate} provides a qualitative comparison of (26) as a function of $z_{i}$ for $n_{1}<n_{2}<n_{3}$ and $\varepsilon_{n}=n^{-1.1}$, so that (\ref{epsilon n}) holds. The plot highlights the contribution of the rescaled state-dependent term of $q_{ij}^{(n)}(z)$ with respect to the one-parameter mutation rate \eqref{symm-mut}.
The behavior of $q_{ij}^{(n)}(z)$ for $z_i$ relatively far from $\varepsilon_{n}$ can be interpreted in terms of a reinforcement mechanism similar to that featured by the $\PD(\theta,\alpha)$ distribution (see \citealp{LMP07b}). It can indeed be observed that the probability that a further sample from (\ref{urn}) is an already observed species is not allocated proportionally to the current frequencies. The ratio of probabilities assigned by \eqref{urn} to any pair of species $(i,j)$ is $r_{i,j}=(n_{i}-\alpha)/(n_{j}-\alpha)$.
When $\alpha\rightarrow0$,
the probability of sampling species $i$ is proportional to the absolute frequency $n_{i}$, or equivalently to $z_{i}$, which in continuous time is reflected by a constant mutation rate as in (\ref{symm-mut}). However, since $r_{i,j}$ is increasing in $\alpha$, a value of $\alpha>0$ reallocates some probability mass from type $j$ to type $i$, so that, for example, for $n_{i}=2$ and $n_{j}=1$ we have $r_{i,j}=2, 3, 5$ for $\alpha=0, 0.5, 0.75$ respectively. Thus $\alpha$ has a reinforcement effect on those species that have higher frequency. 
On the other hand, the behavior of $q_{ij}^{(n)}(z)$ for $z_{i}$ near the boundary is what ultimately makes the process well-defined in a bounded region. For $z_{i}\downarrow\varepsilon_{n}$, \eqref{mut rates} converges to the rate of the one-parameter model \eqref{symm-mut}, so that the associated drift behaves locally as for the one-parameter case and $z_{i}$ is kept inside the boundary $\varepsilon_{n}>0$. Roughly speaking, the decreasing part of the rate function pushes the frequencies towards smaller values, whereas the leftmost plotted part is responsible for keeping the frequencies inside the state space.

The following result shows that the above defined operator characterizes a Feller diffusion on $\TDn$.
Let $||\cdot||$ denote the supremum norm.
\begin{theorem}\label{prop: An}
Let $\A_{n}$ be the operator defined by (\ref{operator An})-(\ref{domain-An})-(\ref{cov-n})-(\ref{drift-n}). Then the closure of $\A_{n}$ in $C(\TDn)$ is single-valued and generates a Feller semigroup $\{\T_{n}(t)\}$ on $C(\TDn)$. For each $\nu_{n}\in\P(\TDn)$, there exists a strong Markov process $Z^{(n)}(\cdot)=\left\{Z^{(n)}(t),t\ge0\right\}$ corresponding to $\{\T_{n}(t)\}$ with initial distribution $\nu_{n}$ and such that
\begin{equation*}
\mathbb{P}\{Z^{(n)}(\cdot)\in C_{\TDn}([0,\infty))\}=1.
\end{equation*} 
\end{theorem}
\begin{proof}
It is easily seen that $\A_{n}$ satisfies the positive maximum principle on $\TDn$, that is if $f\in\D(\A_{n})$ and $z_{0}\in\TDn$ are such that $f(z_{0})=||f||\ge0$, then $\A_{n}f(z_{0})\le 0$. This is immediate in the interior of $\TDn$, while on the boundaries it follows from (\ref{drift at boundary}) and the fact that (\ref{cov-n}) vanishes at every boundary point. Denote now $z^{\sigma}=z_{1}^{\sigma_{1}}\cdots z_{n}^{\sigma_{n}}$ and $\sigma-\delta_{i}=(\sigma_{1},\ldots,\sigma_{i}-1,\ldots,\sigma_{n})$
for $\sigma_{1},\ldots,\sigma_{n}\in\mathbb{N}$. Then 
\begin{align}\label{An mapping}
\A_{n}z^{\sigma}
=&\,\frac{1}{2}\sum_{i}\sigma_{i}\bigg[\frac{\theta}{n-1}z^{\sigma-\delta_{i}}-\frac{\theta n}{n-1}z^{\sigma}
+\frac{2\alpha i}{n(n+1)}\sum_{j=1}^{n}(1-C_{1}e^{-C_{2}z_{j}})z^{\sigma-\delta_{i}}\notag\\
&\,-\alpha z^{\sigma-\delta_{i}}+\alpha C_{1}e^{-C_{2}z_{i}}z^{\sigma-\delta_{i}}\notag\\
&\,+(\sigma_{i}-1)\bigg([1-(n-1)\varepsilon_{n}]z^{\sigma-\delta_{i}}-z^{\sigma}-\varepsilon_{n}[1-(n-1)\varepsilon_{n}]z^{\sigma-2\delta_{i}}+\varepsilon_{n}z^{\sigma-\delta_{i}}\bigg)\bigg]\notag\\
&\,-\frac{1}{2}\sum_{i}\sum_{j\ne i}\sigma_{i}\sigma_{j}
\bigg[z^{\sigma}-\varepsilon_{n}z^{\sigma-\delta_{i}}-\varepsilon_{n}z^{\sigma-\delta_{j}}+\varepsilon_{n}^{2}z^{\sigma-\delta_{i}-\delta_{j}}\bigg]\notag 
\end{align}
for appropriate constants $C_{1},C_{2}$.
Letting $L_{m}$ denote the algebra of polynomials in $(z_{1},\ldots,z_{n})$ restricted to $\TDn$ with degree not greater than $m\in\mathbb{N}$, the image of $\mathcal{A}_{n}$ computed on $L_{m}$ contains functions belonging to $L_{m}$ and of type $e^{-z_{i}}z^{c}$. Since for every $g(x)\in C(K)$, with $K$ compact, and $f(x)=e^{x}g(x)\in C(K)$, there exists a sequence $\{p^{(k)}\}$ of polynomials on $K$ such that $||f-p^{(k)}||\rightarrow0$, so that $||e^{-z}p^{(k)}-g||\rightarrow0$, it follows that the image of $\mathcal{A}_{n}$ is dense in $C(\TDn)$, and so is that of $\lambda-\mathcal{A}_{n}$  for all but at most countably many $\lambda>0$. Since $\cup_{m}L_{m}$ is dense in $C(\TDn)$, the Hille-Yosida Theorem (see \citealp{EK86}, Theorem 4.2.2) now implies that the closure of $\A_{n}$ on $C(\TDn)$ is single-valued and generates a strongly continuous, positive, contraction semigroup $\{\T_{n}(t)\}$ on $C(\TDn)$. The fact that $(1,0)$ belongs to the domain of $\overline\A_{n}$ implies also that $\{\T_{n}(t)\}$ is conservative.  Note now that for every $z_{0}\in\TDn$ and $\delta>0$ there exists $f\in\D(\A_{n})$ such that
\begin{equation*}
\sup_{z\in B^{c}(z_{0},\delta)}f(z)<f(z_{0})=||f||
\qquad \text{and} \qquad 
\A_{n}f(z_{0})=0,
\end{equation*} 
where $B^{c}(z_{0},\delta)$ is a ball of radius $\delta$ centered at $z_{0}$. Take for example $f(z)=-C_{\delta}\sum_{i=1}^{n}(z_{i}-z_{0})^{4}$ for an appropriate constant $C_{\delta}$ which depends on $\delta$.
Then the second assertion follows from Theorem 4.2.7 and Remark 4.2.10 in \cite{EK86}.
\end{proof}

\begin{figure}[t!]
\begin{center}
\includegraphics[width=.5\textwidth]{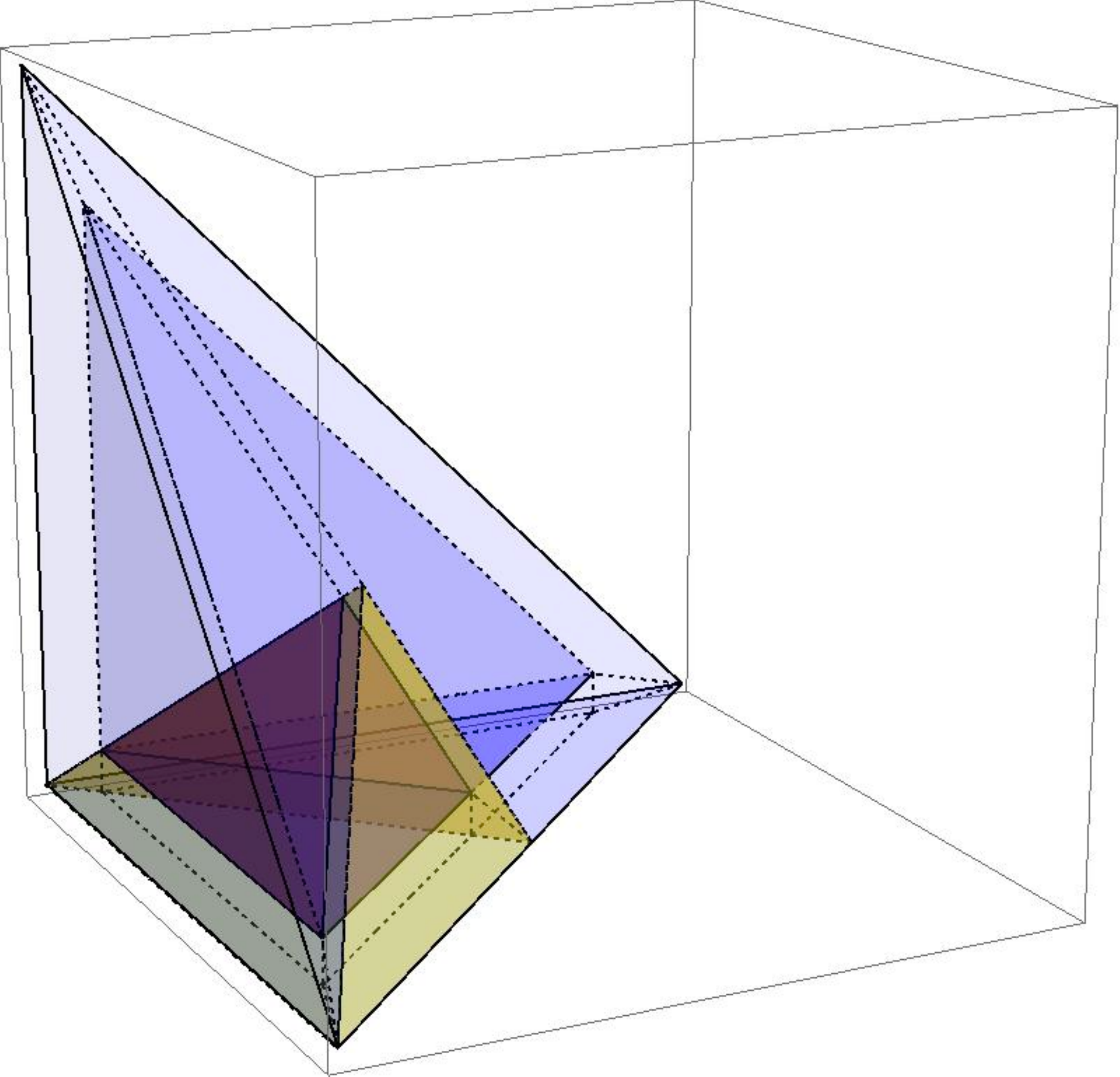}\\
\begin{quote}
\caption{\footnotesize Projection onto the first three coordinates of $\Delta_{n}$ (light blue), $\Delta_{n,\varepsilon_{n}}$ (dark blue), $\nabla_{n,0}$ (yellow), $\nabla_{n,\varepsilon_{n}}$ (red), for $n\ge3$. 
The (ordered) $n$th term of the sequence lives in the red region, so that the frequencies are bounded away from zero for all $n$. As $n$ increases, the red region converges to the yellow region, so that the limit process is left free to move in the full ordered simplex.}\label{fig: simplex}
\end{quote}\end{center}
\end{figure} 

Given $\Ni$ as in (\ref{nabla-infty}), define the subspaces 
\begin{equation}\label{nabla}
\nabla_{\infty}=\Big\{z\in\Ni:\ \sum_{i=1}^{\infty}z_{i}=1\Big\},
\end{equation}
and
\begin{equation*}
\TNn=\Big\{z\in\nabla_{\infty}:\ z_{n}\ge \varepsilon_{n}>z_{n+1}=0\Big\}.
\end{equation*}
See Figure \ref{fig: simplex}.
Define also the Borel measurable map $\rho_{n}:\Delta_{n}\rightarrow\nabla_{\infty}$ by
\begin{equation}\label{rho-n}
\rho_{n}(z)=(z_{(1)},\ldots,z_{(n)},0,0,\ldots),
\qquad \qquad z\in\Delta_{n},
\end{equation} 
where $z_{(i)}$ are the decreasing order statistics of $z\in\Delta_{n}$. It is clear that $\rho_{n}$ maps $\TDn$ into $\TNn$. If $Z^{(n)}(\cdot)$ is the Markov process of Theorem \ref{prop: An}, our aim is thus to show that 
\begin{equation}\label{weak convergence}
\rho_{n}(Z^{(n)}(\cdot))\Rightarrow Z(\cdot)
\end{equation} 
in the sense of convergence in distribution in $C_{\Ni}([0,\infty))$ as $n\rightarrow\infty$, where $Z(\cdot)$ is the diffusion process corresponding to the operator (\ref{operator: theta-sigma}) with an appropriate domain. 
To this end, consider the symmetric polynomials
\begin{equation}\label{symm polyn}
\varphi_{m}(z)=\sum_{i\ge1}z_{i}^{m},\quad \quad z\in\Ni,\ m\ge2,
\end{equation} 
and define
\begin{align}\label{domain B}
\D(\B)
=\Big\{&\text{subalgebra of } C(\Ni)\text{ generated by }
1,\varphi_{3}(z),\varphi_{4}(z),\ldots
\Big\}.
\end{align}

\begin{lemma}\label{lemma: subalgebra dense}
$\D(\B)$ is dense in $C(\Ni)$.
\end{lemma}
\begin{proof}
In \cite{EK81} (see proof of Theorem 2.5) it is proved that the closure of $\D_{0}(\B)$ defined as
\begin{equation}\label{domain B con phi-2}
\D_{0}(\B)
=\Big\{\text{subalgebra of } C(\Ni)\text{ generated by }
1,\varphi_{2}(z),\varphi_{3}(z),\ldots
\Big\}
\end{equation} 
equals $C(\Ni)$. Note now that 
\begin{equation*}
z_{1}=\lim_{m\rightarrow\infty}\varphi_{m}(z)^{1/m},\quad 
z_{2}=\lim_{m\rightarrow\infty}(\varphi_{m}(z)-z_{1}^{m})^{1/m},\quad \ldots
\end{equation*} 
from which
\begin{equation*}
\varphi_{2}(z)=\lim_{m\rightarrow\infty}\varphi_{m}(z)^{2/m}+(\varphi_{m}(z)-z_{1}^{m})^{2/m}+\dots
\end{equation*} 
so that $\varphi_{2}\in\overline{\D(\B)}$. It follows that $\overline{\D_{0}(\B)}\equiv \overline{\D(\B)}$, from which the result follows.
\end{proof}

Before providing the convergence argument, we recall the relevant theorems about the formal existence and the sample path properties of the process $Z(\cdot)$ appearing in (\ref{weak convergence}).

\begin{theorem}\emph{[\citealp*{P09}]}\label{prop: petrov}
Let $\B$ be the operator (\ref{operator: theta-sigma}) with domain (\ref{domain B con phi-2}). The closure of $\B$ in $C(\Ni)$ generates a Feller semigroup $\{\T(t)\}$ on $C(\Ni)$, and for each $\nu\in\P(\Ni)$ there exists a strong Markov process $Z(\cdot)=\left\{Z(t),t\ge0\right\}$ corresponding to $\{\T(t)\}$ with initial distribution $\nu$ and such that
\begin{equation*}
\mathbb{P}\{Z(\cdot)\in C_{\Ni}([0,\infty))\}=1.
\end{equation*} 
\end{theorem}

Let $\nabla_{\infty}$ be as in (\ref{nabla}). The following result shows that if the initial distribution of the Markov process $Z(\cdot)$ of Theorem \ref{prop: petrov} is $\PD(\theta,\alpha)$, then the law of the process is concentrated on $C_{\nabla_{\infty}}([0,\infty))$.

\begin{theorem}\emph{[\citealp*{FS10}]}\label{theorem FS10}
Let $Z(\cdot)=\left\{Z(t),t\ge0\right\}$ be the Markov process of Theorem \ref{prop: petrov}, and assume $Z(0)\sim \PD(\theta,\alpha)$. Then 
\begin{equation*}
\mathbb{P}\{Z(t)\in \nabla_{\infty},\ \forall t\ge0\}=1.
\end{equation*} 
\end{theorem}

Denote now by $\B_{n}$ the right hand side of (\ref{operator An}), with coefficients as in (\ref{cov-n}) and (\ref{drift-n}) but domain 
\begin{equation}\label{domain-Bn}\notag
\D(\B_{n})=\left\{f\in C(\TNn):\ f\circ\rho_{n}\in C^{2}(\TDn)\right\}.
\end{equation} 
We are now ready to state the main result of the section.

\begin{theorem}\label{theorem: convergence}
Let $Z^{(n)}(\cdot)$ and $Z(\cdot)$ be the Markov processes of Theorem \ref{prop: An} and Theorem \ref{prop: petrov}, with initial distribution $\nu_{n}\in\P(\TDn)$ and $\nu\in\P(\Ni)$, respectively. If $\nu_{n}\circ\rho_{n}^{-1}\Rightarrow \nu$, then (\ref{weak convergence}) holds in $C_{\Ni}([0,\infty))$. If in addition $\nu$ is $\PD(\theta,\alpha)$, then (\ref{weak convergence}) holds in $C_{\nabla_{\infty}}([0,\infty))$.
\end{theorem}
\begin{proof}
For $\rho_{n}$ as in (\ref{rho-n}), define $\pi_{n}:C(\Ni)\rightarrow C(\Delta_{n})$ by $\pi_{n}f=f\circ\rho_{n}$ and note that $\pi_{n}:\D(\B)\rightarrow\D(\A_{n})$. Since
for every $f\in\D(\B)$ we have $\pi_{n}\B f=\B(f\circ\rho_{n})$, 
for all such functions and $z\in\TDn$ we have 
\begin{align*}
\mathcal{A}_{n}&\pi_{n}f(z)-\pi_{n}\mathcal{B}f(z)=\\
=&\,\frac{1}{2}\sum_{i,j=1}^{n}
\Big[a_{ij}^{(n)}(z)-z_{i}(\delta_{ij}-z_{j})\Big]
\frac{\partial^{2}f(\rho_{n}(z))}{\partial z_{i}\partial z_{j}}
+\sum_{i=1}^{n}
\Big[b_{i}^{(n)}(z)+\theta z_{i}+\alpha\Big]
\frac{\partial f(\rho_{n}(z))}{\partial z_{i}}.
\end{align*}
It can be easily verified that the absolute value of the first term in square brackets is bounded above by a term of order $O(n\varepsilon_{n})$, and that 
\begin{equation}\label{gener inequ}
\begin{split}
|\mathcal{A}_{n}&\pi_{n}f(z)-\pi_{n}\mathcal{B}f(z)|\\
&\le O(n\varepsilon_{n})\sum_{i,j=1}^{n}\bigg|\frac{\partial^{2} f(\rho_{n}(z))}{\partial z_{i}\partial z_{j}}\bigg|
+O(n^{-1})\sum_{i=1}^{n}\bigg|\frac{\partial f(\rho_{n}(z))}{\partial z_{i}}\bigg|\\
&\,\quad +O(n^{-1})\sum_{i=1}^{n}i
\bigg|\frac{\partial f(\rho_{n}(z))}{\partial z_{i}}\bigg|
+\alpha\sum_{i=1}^{n}\exp\left\{-2(z_{i}-\varepsilon_{n})/\varepsilon_{n}\right\}\bigg|\frac{\partial f(\rho_{n}(z))}{\partial z_{i}}\bigg|.
\end{split}
\end{equation}
Observe now that for $f\in\mathscr{D}(\mathcal{B})$ of type $\varphi_{m_{1}}\times\cdots\times\varphi_{m_{k}}$, we have $f(\rho_{n}(z))=f(z)$ and
\begin{align}\label{sum of derivatives}
\sum_{i=1}^{n}\left|\frac{\partial f(z)}{\partial z_{i}}\right|
=&\, \sum_{i=1}^{n}\sum_{j=1}^{k}m_{j}z_{i}^{m_{j}-1}\prod_{h\ne j}\varphi_{m_{h}}
\le \sum_{j=1}^{k}m_{j}\sum_{i=1}^{n}z_{i}^{m_{j}-1},
\end{align}
which is bounded above by $\sum_{j=1}^{k}m_{j}<\infty$. Let $m_{j}=3$ (cf.~\eqref{domain B}). Then \eqref{sum of derivatives} implies
\begin{align*}
\sum_{i=1}^{n}&\exp\left\{-2(z_{i}-\varepsilon_{n})/\varepsilon_{n}\right\}\bigg|\frac{\partial f(z)}{\partial z_{i}}\bigg|
\le 
n\varepsilon_{n}^{2}\sum_{j=1}^{k}m_{j}\rightarrow0
\end{align*}
uniformly as $n\rightarrow\infty$, where we have used (\ref{epsilon n}) and the fact that, for $f$ as above, the left hand side is maximized when $z_{i}=\varepsilon_{n}$.
Furthermore,
\begin{align*}
O(n^{-1})\sum_{i=1}^{n}i\bigg|\frac{\partial f(z)}{\partial z_{i}}\bigg|
\le O(n^{-1})\sum_{j=1}^{k}m_{j}\sum_{i=1}^{n}iz_{i}^{m_{j}-1}\to0
\end{align*}
for $m_{j}\ge3$ since $z_{i}\le i^{-1}$.
Finally
\begin{equation}\label{second partials}
\begin{split}
\sum_{i,j=1}^{n}&\,\left|\frac{\partial^{2} f(z)}{\partial z_{i}\partial z_{j}}\right|
\le \sum_{i,j=1}^{\infty}\Bigg[\partial_{ij}\varphi_{m_{h}}
\prod_{\ell\ne h} \varphi_{m_{\ell}}+\sum_{q\ne h}\partial_{i}\varphi_{m_{h}}\partial_{j}\varphi_{m_{q}}\prod_{\ell\ne h,q} \varphi_{m_{\ell}}\Bigg]\\
=&\,  \Bigg[m_{h}(m_{h}-1)\varphi_{m_{h}-2}
\prod_{\ell\ne h} \varphi_{m_{\ell}}
+\sum_{q\ne h}m_{h}m_{q}\varphi_{m_{h}+m_{q}-2}\prod_{\ell\ne h,q} \varphi_{m_{\ell}}\\
&\,+\sum_{q\ne h}m_{h}m_{q}\varphi_{m_{h}-1}\varphi_{m_{q}-1}\prod_{\ell\ne h,q} \varphi_{m_{\ell}}
\Bigg]\notag\\
\le &\,
\Bigg[m_{h}(m_{h}-1)+\sum_{q\ne h}m_{h}m_{q}+\sum_{q\ne h}m_{h}m_{q}
\Bigg]
\end{split}
\end{equation}
whose right hand side is bounded. From \eqref{gener inequ}, the above arguments and (\ref{epsilon n}) imply that 
\begin{equation}\label{eq:Bn-convergence}
||\mathcal{A}_{n}\pi_{n}f-\pi_{n}\mathcal{B}f||\longrightarrow0,\notag
\qquad\qquad f\in\D(\B).
\end{equation}
Given now that $\rho_{n}(Z^{(n)})$ clearly satisfies a compact containment condition (cf.~\citealp{EK86}, Remark 3.7.3), and that Theorem \ref{prop: petrov} implies that the closure of $\mathcal{B}$ in $C_{\Ni}([0,\infty))$ generates a strongly continuous contraction semigroup, 
the hypotheses of Corollary 4.8.7 in \cite{EK86} are satisfied and 
the first assertion of the Theorem follows with $C_{\Ni}([0,\infty))$ replaced by $D_{\Ni}([0,\infty))$, the space of \emph{c\`adl\`ag} functions from $[0,\infty)$ to $\Ni$.
Furthermore, the convergence holds in $C_{\overline\nabla_{\infty}}([0,\infty))\subset D_{\overline\nabla_{\infty}}([0,\infty))$, since the limit probability measure is concentrated on $C_{\overline\nabla_{\infty}}([0,\infty))$ and the Skorohod topology relativized to $C_{\overline\nabla_{\infty}}([0,\infty))$ coincides with the topology on $C_{\overline\nabla_{\infty}}([0,\infty))$. See for example \cite{B68}, Section 18.
Finally, the second assertion follows from Theorem \ref{theorem FS10} and from further relativization to $C_{\nabla_{\infty}}([0,\infty))$ of the topology on $C_{\Ni}([0,\infty))$.
\end{proof}


\section*{Acknowledgements}

Research supported by the European Research Council (ERC) through StG "N-BNP" 306406. The author is grateful to an anonymous Referee and to Pierpaolo De Blasi for carefully reading the manuscript and for several useful suggestions which lead to an improvement of the presentation.


\end{document}